\documentclass[11pt]{article}

\usepackage{amsmath}

\usepackage{times}
\usepackage{bm}

\usepackage[plain,noend]{algorithm2e}

 \usepackage{multirow,verbatim}
 \usepackage{latexsym, amsbsy, amssymb}
 \usepackage{natbib}
 \usepackage{graphicx}
 \usepackage{booktabs}
 \usepackage{bbm}
 \usepackage{color}

 \usepackage{float} 
 \usepackage{hyperref}
\hypersetup{colorlinks=true,linkcolor=blue,urlcolor=blue,citecolor=blue}

\usepackage[caption=false,listofformat=subsimple,labelformat=simple]{subfig}

\def\T{{ \mathrm{\scriptscriptstyle T} }}

\DeclareMathOperator*{\argmin}{argmin}

\def\var{\mathrm{var}}

\newtheorem{proposition}{{\bf Proposition}}

\newtheorem{definition}{{\bf Definition}}
\newtheorem{theorem}{{\bf Theorem}}

\newcommand\ca[1]{{\cal{#1}}}
\newcommand\lo[1]{_{\nano{#1}}}
\newcommand\hi[1]{^{\nano{#1}}}
\newfont{\rsfsten}{rsfs10 scaled 1100}
\newfont{\rsfstena}{rsfs10 scaled 800}
\newfont{\rsfstenb}{rsfs10 scaled 800}

\def\L{{\cal L}}

\def\var{\mathrm{var}}

\def\ran{\mathrm{ran}}

\def\nano{\scriptscriptstyle}

\def\L2T{L \lo 2 (T)}
\def\L2TX{L \lo 2 (T\lo X)}
\def\L2TX{L \lo 2 (T\lo Y)}

\begin{document}

\title{\bf On Exact Bayesian Credible Sets for Classification and Pattern Recognition}
\author{Chaegeun Song and Bing Li\\ 
    Department of Statistics, The Pennsylvania State University\\
    cqs6027@psu.edu \quad bxl9@psu.edu}

\date{}   

\maketitle

\vspace{-18pt}

\maketitle


\vspace{.0in}
\newcommand\red[1]{{\color{red}{#1}}}


\begin{abstract}
The current definition of a Bayesian credible set cannot, in general, achieve an arbitrarily preassigned credible level. This drawback is particularly acute for classification problems, where there are only a finite number of achievable credible levels. As a result, there is as of today no general way to construct an exact credible set for classification. In this paper, we introduce a generalized credible set that can achieve any preassigned credible level. The key insight is a simple connection between the Bayesian highest posterior density credible set and the Neyman--Pearson lemma, which, as far as we know, hasn't been noticed before. Using this connection, we introduce a randomized decision rule to fill the gaps among the discrete credible levels. Accompanying this methodology, we also develop the Steering Wheel Plot to represent the credible set, which is useful in visualizing the uncertainty in classification. By developing the exact credible set for discrete parameters, we make the theory of Bayesian inference more complete. 
\end{abstract}

\medskip
\medskip

\newcommand\acuteout[1]{$\acute{\rm{#1}}$}
\newcommand\ddotout[1]{$\ddot{\rm{#1}}$}
\def\ali{&\,}
\newcommand\seq[1]{\{ {#1} \lo k \}}
\def\spn{{\rm{span}}}
\def\nil{\hi {0}}
\def\tsum{\textstyle{\sum}}
\def\tsum{{\textstyle{\sum}}}
\def\cclass{\frak{M} _ {\ca G \lo {Y|X}}}
\def\spn{{\rm span}}
\def\one{{\mathbbm 1}}
\def\oc{\hi {\perp}}
\def\cran{\overline{\ran}}
\def\cran{\overline{\ran}}
\def\mpinv{\hi {\dagger}}
\def\cran{\overline{\ran}}
\def\adj{\hi *}


\section{Introduction}

The credible set is one of the most useful tools in Bayesian inference to quantify the uncertainty in Bayesian estimation. See, for example, \citet{berger1985statistical}, and \citet{li2019graduate}. It plays a similar role as the confidence set in the frequentist setting. The most commonly used credible set is the highest posterior density credible set, which has the shortest length for a given credible level. 

However, the classical definition of the highest posterior density credible set \citep{berger1985statistical} is useful mostly for continuous parameters\textemdash the parameter supported on an interval or a set with a nonempty and connected interior. In the discrete case where the parameter is supported on a finite or a countable set, the highest posterior density credible set may not exist for a prescribed credible level. This problem is particularly acute for Bayesian classification where the highest posterior density credible set almost never exists for a given credible level. To our knowledge, there has been no systematic method to construct the shortest credible set for such problems. 

Recent studies on the credible set have focused on frequentist coverage probability rather than the credible level. \citet{fang2006empirical} developed the asymptotic frequentist coverage of credible sets. See also \citet{hutchinson2020improving}. The credible intervals over a constrained range and their frequentist coverage probabilities have been studied by \citet{roe2000setting}, \citet{zhang2003credible} and \citet{marchand2013bayesian}.  \citet{szabo2015frequentist} worked on frequentist coverage of adaptive nonparametric credible sets. \citet{belitser2017coverage} constructed credible balls, and \citet{syring2019calibrating} considered calibration method to achieve prescribed coverage probability. 

Bayesian classification has been widely used in many contemporary scientific studies such as cancer research, genomics, and healthcare applications. See, for example, \citet{mallick2005bayesian}, \citet{sambo2012bag},\citet{bhuvaneswari2012naive},\citet{knight2014mcmc}, \citet{tavtigian2018modeling}, and \citet{ibeni2019comparative}. Thus, the need for precise uncertainty quantification for Bayesian classification is prevalent. It is high time to facilitate Bayesian classification with exact inference tools and bring it up to par with Bayesian estimation for continuous parameters, where exact credible sets have long been available and widely used. The current paper is one step towards satisfying this need. 

Our extension is based on the following insight: the process of identifying the highest posterior density credible set is rather similar to that of constructing the most powerful test based on the Neyman--Pearson lemma, and the latter can accommodate any prescribed significance level\textemdash even for discrete random variables\textemdash by strategically placing probability masses at some transitioning points. This consideration leads us to define the generalized highest posterior density credible set not as a set in parameter space, but as a mapping from the parameter space to the closed interval $[0,1]$. By this adjustment, there always exists a unique shortest credible set for any prescribed credible level, regardless of the nature of the parameter. Accompanying this methodology, we also introduce a visualization tool for classification, called Steering Wheel Plot.

\section{Generalized credible set}

Suppose $(\Omega, \mathcal{F}, P)$ is a probability space, and $(\Omega_X, \mathcal{F}_X, \mu_X)$, $(\Omega_\Theta, \mathcal{F}_\Theta, \mu_\Theta)$ are $\sigma$-finite measure spaces. Suppose $X: \Omega \to \Omega_X$ is a random element measurable with respect to $\mathcal{F}/ \mathcal{F}_X$, and $\Theta: \Omega \to \Omega_\Theta$ a random element measurable with respect to $\mathcal{F} / \mathcal{F}_\Theta$. For convenience, we assume that the pair $(X, \Theta)$ takes values in the product measure space $(\Omega_X \times \Omega_\Theta, \mathcal{F}_X \times \mathcal{F}_\Theta, \mu_X \times \mu_\Theta)$. Here, the random element $X$ represents the data, and the random element $\Theta$ represents the parameter. We will later focus on the case where $\Theta$ is supported on a finite set (i.e. cardinality of $\Omega_\Theta$ is finite), but there is no need to make such an assumption to state our general result.

 Let
\begin{align*}
P_\Theta = P \circ \Theta^{-1}, \quad P_X = P \circ X^{-1}, \quad P_{X\Theta} = P \circ (X, \Theta)^{-1}
\end{align*}
be the distributions $\Theta$, $X$, and $(X, \Theta)$, respectively. Assume that these distributions are dominated by $\mu_\Theta$,  $\mu_X$ and $\mu_X \times \mu_\Theta$, respectively. 
 Let $\pi_\Theta$, $f_X$ and $f_{X \Theta}$ be the density of $P_\Theta$, $P_X$ and $P_{X\Theta}$ with respect to their dominating measures. Let
\begin{align*}
\pi_{\Theta \mid X} (\theta \mid x)
=
\begin{cases}
f_{X\Theta} (x, \theta)/ f_X(x) & \text{if} \quad f_X(x)  > 0 \\
0 & \text{if} \quad f_X (x) = 0, 
\end{cases}
\\
f_{X \mid \Theta} (x \mid \theta )
=
\begin{cases}
f_{X\Theta} (x, \theta)/ \pi_\Theta (\theta)  & \text{if} \quad \pi_\Theta (\theta)   > 0 \\
0 & \text{if} \quad \pi_\Theta (\theta) = 0.
\end{cases}
\end{align*}
In the usual Bayesian terminology, $\pi_\Theta$ is the prior density of $\Theta$, $f_X$ is the marginal density of $X$, $f_{X \Theta}$ is the joint density of $(X, \Theta)$, $\pi_{\Theta \mid X}$ is the posterior density, and $f_{X \mid \Theta}$ is the likelihood.

The classical definitions of a credible set and the highest posterior density credible set proceed as follows (See, for example, \citealp{berger1985statistical}, page 140; and  \citealp{li2019graduate}, page 180). For an observed value $x$ of $X$ and $0< \alpha < 1$, a $100(1-\alpha)$\% credible set for $\Theta$ is any $C \in \mathcal{F}_\Theta$ such that 
\begin{align*}
    P\{\Theta^{-1} (C) \mid X = x\} \geq 1-\alpha.
\end{align*} 
To define the $100(1-\alpha)$\% highest posterior density credible set for $\Theta$, let 
\begin{align*}
    C(\kappa) = \{ \theta \in \Omega_\Theta : \pi_{\Theta \mid X} (\theta \mid x) \geq \kappa \}
\end{align*}
for $\kappa \geq 0$. The set-valued function $C(\kappa)$ is decreasing in the sense that if $\kappa_1 < \kappa_2$, then $C(\kappa_2) \subseteq C(\kappa_1)$. Let $\kappa_\alpha = \sup \{ \kappa : P \{ C(\kappa) \mid x\} \geq 1-\alpha \}$. Then the $100(1-\alpha)$\% highest posterior density credible set is defined as $C(\kappa_\alpha)$. It can be shown that the $\mu_\Theta$-measure of $C(\kappa_\alpha)$ is minimal among any $100(1-\alpha)$\% credible set. See the Theorem 6.4 in \cite{li2019graduate}.

\begin{figure}[]
\centering
\subfloat[]{%
  \includegraphics[width=0.45\textwidth]{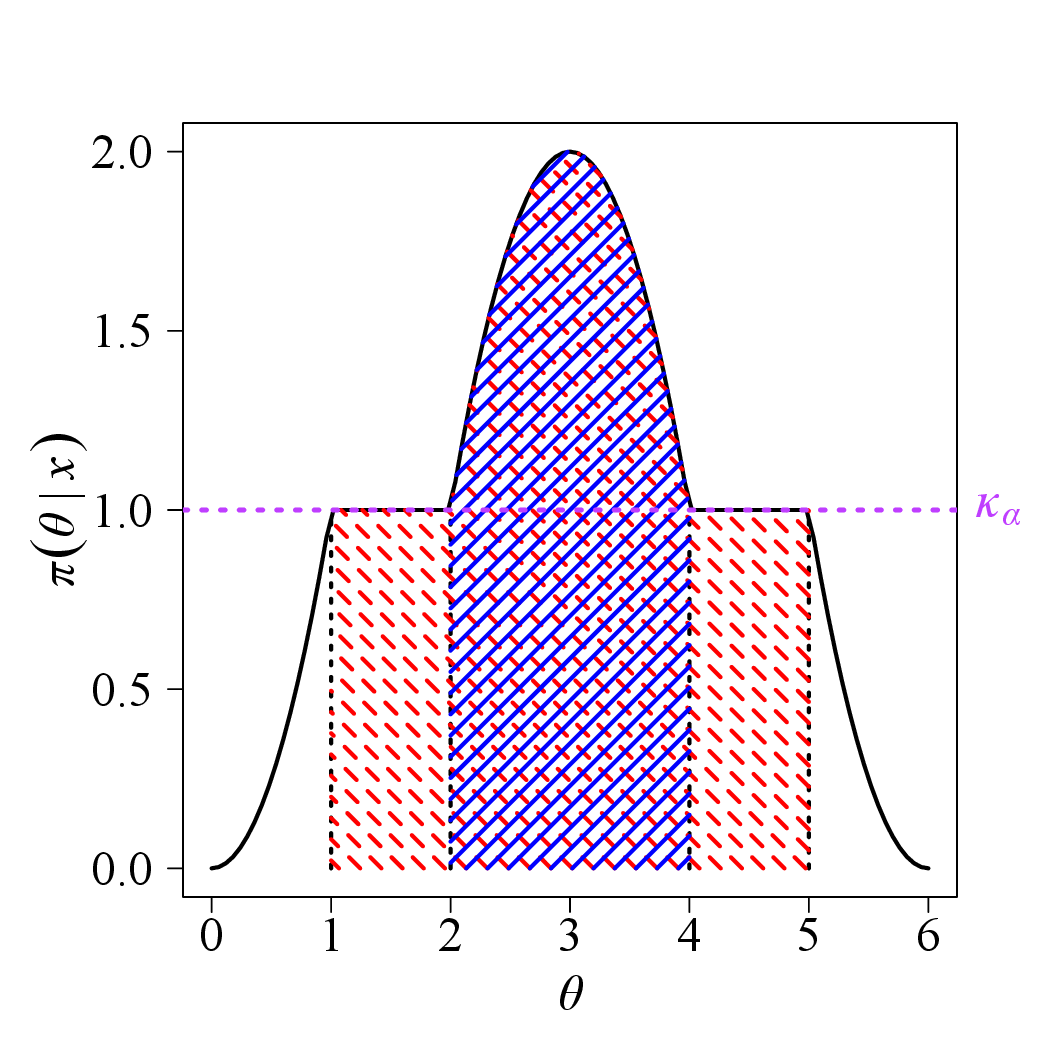}\label{fig:flatdistn}%
}%
\subfloat[]{%
  \includegraphics[width=0.45\textwidth]{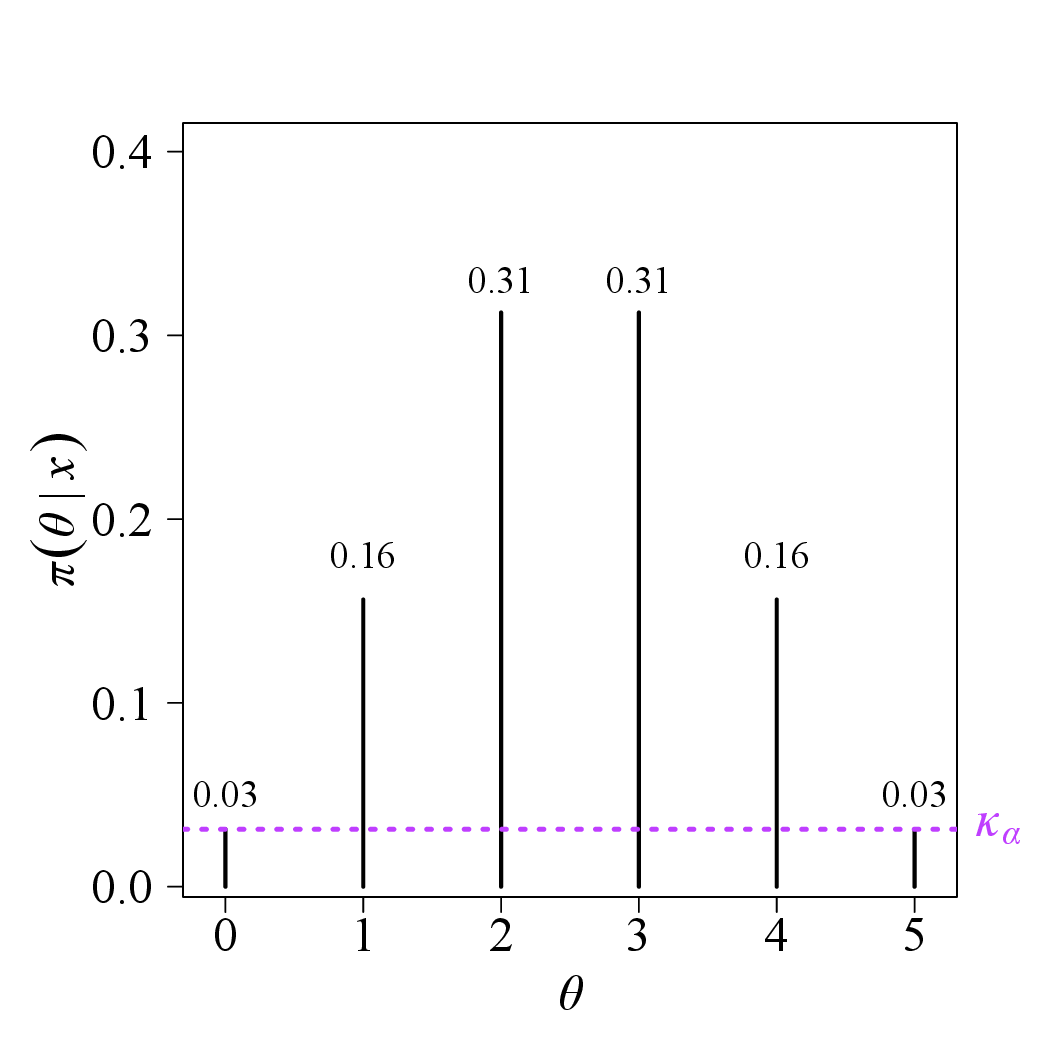}\label{fig:discdistn}%
}
\caption{Failure of having exact $100(1-\alpha)$\% highest posterior density credible set. (a) Posterior density function having flat spots with $\alpha=0.35$ so that $1-\alpha$ is larger than the blue-shaded area (solid) but smaller than the red-shaded area (dashed); 
(b) Posterior mass function with $\alpha=0.05$.
}
\label{fig:highest posterior densityfail}
\end{figure}

To see why $C(\kappa_\alpha)$ doesn't always exist, consider the posterior density 
\begin{align*}
   \pi (\theta \mid x ) =(3/20)[\theta^2\mathbbm{1}_{[0,1]} (\theta)+\mathbbm{1}_{\{(1,2) \cup (4,5)\}} (\theta)-\{(\theta-3)^2+2\}\mathbbm{1}_{[2,4]}+(\theta-6)^2\mathbbm{1}_{[5,6]} (\theta)] 
\end{align*}
where $\mathbbm{1}_A(\cdot)$ is an indicator function of set $A$ shown in Fig.~\ref{fig:flatdistn}. Suppose $\alpha=0.35$. Then, $\kappa_\alpha = 1$ corresponds to the flat spot in the posterior density. However, since 
\begin{align*}
    P\{ C(\kappa_\alpha) \mid x\} >  1-\alpha, 
\end{align*}
the prescribed credible level is not met by the highest posterior density credible set. 

This situation occurs most often when $\theta$ is discrete. Figure~\ref{fig:discdistn} shows the posterior mass function of a discrete random parameter, supported on a finite set. Specifically, consider the scenario where $\Theta \mid X \sim \mathrm{Bin}(5, 1/2)$. Assume that $\alpha = 0.05$ and we are interested in 95\% highest posterior density credible set. Here, $\kappa_\alpha = 0.03125$ is represented by the dashed line, which is of the same height as the posterior mass function at $\theta = 0$ or $\theta = 5$. The highest posterior density credible set $C(\kappa_\alpha)$ is the entire parameter space $\{ 0, 1, 2, 3, 4, 5 \}$ whose credible level is 100\%. Again, it fails to attain the exact 95\% credible level.  

These counterexamples lead us to define the following generalized credible set. Henceforth, for a function $\phi: \Omega_\Theta  \times \Omega_X \to [0,1]$, we write $\phi ( \theta, x)$ as $\phi (\theta \mid x)$ to emphasize the conditional nature of the variable $X$.

\begin{definition} \label{def:GCS}
For an $\alpha \in (0,1)$, we call any measurable function $\phi: \Omega_\Theta \times \Omega_X \to [0,1]$ such that
\begin{align*}
    E\{ \phi (\Theta \mid X) \mid X \}= 1 - \alpha \quad \text{a.s.}\quad P_X
\end{align*}
a $(1-\alpha)$-level generalized credible set for $\theta$.
\end{definition} 

The generalized credible set is a mapping from $\Omega_\Theta \times \Omega_X$ to the interval $[0, 1]$. In contrast, the classical credible set, which is a subset of $\Omega_\Theta$, can be viewed as a mapping $\phi$ from $\Omega_\Theta \times \Omega_X$ to the set $\{ 0, 1\}$; that is, $\phi$ is the indicator function of the credible set. Since $\{ 0, 1\} \subseteq [0, 1]$, the generalized credible set indeed generalizes the credible set. 

\section{Generalized highest posterior density credible set}
As mentioned earlier,  the highest posterior density credible set is the smallest credible set in the sense that it has the smallest $\mu_\Theta$-measure among all credible sets of the same level. More precisely, if $C ( \kappa_\alpha)$ is the highest posterior density credible set of level $1-\alpha$ and $C$ is any other credible set of level $1-\alpha$, then $\mu_{\Theta} \{C ( \kappa_\alpha)\} \le \mu_\Theta ( C)$. In this section, we extend this optimal property to generalized credible sets. We first generalize the highest posterior density credible set.

\begin{definition}\label{def:GHPD}
If there is a constant $\kappa_\alpha (x) > 0$ such that the generalized credible set 
\begin{equation}\label{eq:GHPD}
\begin{aligned}
\phi^* ( \theta \mid x ) = 
\begin{cases}
1 & \mbox{if} \ \pi_{\Theta \mid X} (\theta \mid x) > \kappa_\alpha (x) \\
\gamma (\theta \mid x) & \mbox{if} \ \pi_{\Theta \mid X} (\theta \mid x) = \kappa_\alpha (x) \\
0 & \mbox{if} \ \pi_{\Theta \mid X} (\theta \mid x) < \kappa_\alpha (x)
\end{cases}
\end{aligned}
\end{equation}
satisfies $E\{ \phi^* (\Theta \mid X) \mid X \} = 1-\alpha$, then $\phi^*$ is called the level $1-\alpha$ generalized highest posterior density credible set, or $(1-\alpha)$-level generalized highest posterior density credible set. 
\end{definition}

If we take $\gamma (\theta \mid x)$ to be constantly 1, the above function is nothing but the indicator function of the set $\{ \theta: \pi_{\Theta \mid X} ( \theta \mid x) \ge \kappa_\alpha \}$, which is the definition of the highest posterior density credible set. The difference is that for the highest posterior density credible set we can only guarantee $P [ \Theta^{-1} \{ C ( \kappa_\alpha ) \} \mid X] \ge 1- \alpha$, where the equality may not be achieved. In comparison, as we will show in Theorem \ref{thm:exuniq}, the exact level $1-\alpha$ can always be achieved by the generalized highest posterior density credible set. 

Now that the measure $\mu_\Theta (C)$ is replaced by the integral $\int \phi (\theta \mid x) d \mu_\Theta (\theta)$, the size of a generalized credible set has a new meaning: we simply use $\int \phi (\theta \mid x) d \mu_\Theta (\theta)$ to define the size of a generalized credible set $\phi (\theta \mid x)$. The next two theorems show that the generalized highest posterior density credible set is the smallest generalized credible set, that it exists, and that it is unique. The proofs of the theorems echo the proof of the Neyman--Pearson lemma. 

\begin{theorem}\label{thm:smallest}
If $\phi^*: \Omega_\Theta \times \Omega_X \to [0,1]$ is the level $1-\alpha$ generalized highest posterior density credible set, then it is the smallest level $1-\alpha$ generalized credible set; that is, for any $\phi: \Omega_\Theta \times \Omega_X \to [0,1]$ satisfying $E \{ \phi (\Theta \mid X ) \mid X \} = 1- \alpha$, we have 
\begin{align*}
\int \phi^* ( \theta \mid x) d \mu_\Theta (\theta) \le \int \phi ( \theta \mid x) d \mu_\Theta (\theta). 
\end{align*}
\end{theorem}

\begin{proof}
Let $\phi$ be any $(1 - \alpha)$-level generalized credible set. Then, by expression \eqref{eq:GHPD}, 
\begin{align*}
\{\phi^* (\theta \mid x) - \phi (\theta \mid x)\}\{\pi_{\Theta \mid X} ( \theta \mid x) - \kappa_\alpha (x)\} \ge 0 
\end{align*}
for all $\theta \in \Omega _\Theta$ and $x \in \Omega _X$. Therefore,
\begin{align*}
\int_{\Omega_\Theta} \{\phi^* (\theta \mid x) - \phi (\theta \mid x)\}\{\pi_{\Theta \mid X} ( \theta \mid x) - \kappa_\alpha (x)\} d \mu_\Theta (\theta)  \ge 0.
\end{align*}
It follows that
\begin{align*}
\kappa_\alpha (x) \int_{\Omega_\Theta} \{\phi^* (\theta \mid x) - \phi (\theta \mid x)\} d \mu_\Theta (\theta)
\le \int_{\Omega_\Theta} \{\phi^* (\theta \mid x) - \phi (\theta \mid x)\} \pi_{\Theta \mid X} ( \theta \mid x)d \mu_\Theta (\theta) = 0.
\end{align*}
Because $\kappa_\alpha(x) > 0$, we have the desired result.
\end{proof}

We next show the existence of the generalized highest posterior density credible set and the uniqueness of the smallest generalized credible set. 

\begin{theorem} \label{thm:exuniq}
For any $ 0 < \alpha < 1$ and $x \in \Omega_X$, there exists a $(1-\alpha)$-level generalized highest posterior density credible set $\phi^*$ with $\gamma ( \theta \mid x)=\gamma (x)$. Furthermore, if $\phi^\prime$ is any smallest generalized credible set of level $1-\alpha$, then it has the form \eqref{eq:GHPD} a.s. $P_X$. That is, 
\begin{align*}
    P_{\Theta \mid X} \{ \phi^* (\theta \mid x) \neq \phi^\prime (\theta \mid x),\; \pi_{\Theta \mid X} ( \theta \mid x) \neq \kappa_\alpha (x) \mid X =x \} =0 \quad a.s. \quad P_X.
\end{align*}
\end{theorem}

\begin{proof}
(Existence) Define
\begin{align*}
    \rho (\kappa \mid x) = P_{\Theta \mid X} \{ \pi_{\Theta \mid X}(\theta \mid x) > \kappa \mid x \}.
\end{align*}
By continuity of probability, it is easy to show that $\rho (\kappa \mid x)$ is a nonincreasing, right continuous function of $\kappa$ with left limit $\rho (\kappa - \mid x)= P_{\Theta \mid X} \{ \pi_{\Theta \mid X}(\theta \mid x) \geq \kappa \mid x \}$. Also, we have
\begin{align*}
    \rho (0 - \mid x)=1, \quad \rho(0 \mid x)=P_{\Theta \mid X} \{ \pi_{\Theta \mid X}(\theta \mid x) > 0 \mid x \}, \quad \lim_{\kappa \to \infty} \rho (\kappa \mid x) = 0.
\end{align*}
By proposition 3.1 of \cite{li2019graduate}, for any $0< \alpha < 1$, there exists a $0 < \kappa_\alpha (x) < \infty$ such that 
\begin{align*}
    \rho \{ \kappa_\alpha (x) \mid x \} \leq 1 - \alpha \leq \rho \{ \kappa_\alpha (x) - \mid x \}.
\end{align*}

If we define $\gamma (x)$ as 
\begin{align*}
    \gamma (x)= 
    \begin{cases}
    \cfrac{1 - \alpha - P_{\Theta \mid X} \{ \pi_{\Theta \mid X}(\theta \mid x) > \kappa_\alpha (x) \mid x \} }{P_{\Theta \mid X} \{ \pi_{\Theta \mid X}(\theta \mid x) = \kappa_\alpha (x) \mid x \} } & \text{if} \quad P_{\Theta \mid X} \{ \pi_{\Theta \mid X}(\theta \mid x) = \kappa_\alpha (x) \mid x \} > 0 \\
    0 & \text{otherwise}, 
    \end{cases}
\end{align*}
then $E_{\Theta \mid X} \{ \phi^*  (\Theta \mid x) \mid X = x \} = 1-\alpha$.
\\ 
\indent (Uniqueness)
Fix $0< \alpha <1$. Let $\phi^* $ be the smallest $(1-\alpha)$-level generalized credible set of the form \eqref{eq:GHPD} and $\phi^\prime $ be another the smallest $(1-\alpha)$-level generalized credible set. Since their credible levels are $1 - \alpha$, we have
\begin{align*}
    E_{\Theta \mid X} \{ \phi^*  (\Theta \mid x) \mid X = x \} = E_{\Theta \mid X} \{ \phi^\prime  (\Theta \mid x) \mid X = x \} = 1 - \alpha, 
\end{align*}
which implies
\begin{align}\label{eq:post prob}
    \int_{\Omega_\Theta} \{\phi^* (\theta \mid x) - \phi^\prime (\theta \mid x)\} \pi_{\Theta \mid X} ( \theta \mid x) d \mu_\Theta (\theta)=0.
\end{align}
Since they are the smallest credible sets, we also have 
\begin{align}\label{eq:smallest}
    \int_{\Omega_\Theta} \{\phi^* (\theta \mid x) - \phi^\prime (\theta \mid x)\} d \mu_\Theta (\theta)=0.
\end{align}
By \eqref{eq:post prob} and \eqref{eq:smallest}, we have
\begin{align*}
    & \int_{\Omega_\Theta} \{\phi^* (\theta \mid x) - \phi^\prime (\theta \mid x)\} \{\pi_{\Theta \mid X} ( \theta \mid x) - \kappa_\alpha (x)\} d \mu_\Theta (\theta)\\
    &=\int_{\Omega_\Theta} \{\phi^* (\theta \mid x) - \phi^\prime (\theta \mid x)\} \pi_{\Theta \mid X} ( \theta \mid x) d \mu_\Theta (\theta) - \kappa_\alpha (x) \int_{\Omega_\Theta} \{\phi^* (\theta \mid x) - \phi^\prime (\theta \mid x)\} d \mu_\Theta (\theta)\\
    &=0
\end{align*}
Since $\{\phi^* (\theta \mid x) - \phi^\prime (\theta \mid x)\} \{\pi_{\Theta \mid X} ( \theta \mid x) - \kappa_\alpha (x)\} \geq 0$, it is 0 a.e. with respect to $\mu_\Theta$. This implies 
\begin{align*}
    \mu_\Theta  \{ \phi^ * (\theta \mid x ) \neq \phi^\prime (\theta \mid x),\; \pi_{\Theta \mid X} ( \theta \mid x) \neq \kappa_\alpha (x) \mid x \}  = 0.
\end{align*}
By $P_{\Theta \mid X} \ll \mu_\Theta$, we have
\begin{align*}
    P_{\Theta \mid X} \{ \phi^* (\theta \mid x ) \neq \phi^\prime (\theta \mid x ),\; \pi_ {\Theta \mid X} ( \theta \mid x) \neq \kappa_\alpha (x) \mid x \}=0 \quad a.s. \; P_X, 
\end{align*}
which is the desired result.
\end{proof}

Let
$\rho (\kappa \mid x)$ be as defined in the proof of Theorem \ref{thm:exuniq}. The next proposition shows that the generalized highest posterior density credible set reduces to the highest posterior density credible set when $\rho(\cdot \mid x)$ is continuous at $\kappa_\alpha$.

\begin{proposition} \label{prop:highest posterior density}
Let $\kappa_\alpha (x) = \sup \{ \kappa : P \{ \pi_{\Theta \mid X} ( \theta \mid x ) \geq \kappa \mid x \} \geq 1-\alpha \}$. If $\rho$ is continuous at $\kappa_\alpha (x)$, then the generalized highest posterior density credible set $\phi^*$ in \eqref{eq:GHPD} reduces to 
\begin{equation*}\label{eq:highest posterior density}
\begin{aligned}
\phi^* (\theta \mid x) =
\begin{cases}
1 & \text{if} \quad \pi_{\Theta \mid X} ( \theta \mid x) \geq \kappa_\alpha (x)\\
0 & \text{if} \quad \pi_{\Theta \mid X} ( \theta \mid x) < \kappa_\alpha (x)
\end{cases}
\end{aligned}
\end{equation*}
so that $\{ \theta : \phi^* (\theta \mid x) = 1 \}$ is the $(1-\alpha)$-level highest posterior density credible set. 
\end{proposition}
\begin{proof}
Let $S = \{ \kappa : P \{ \pi_{\Theta \mid X} ( \theta \mid x ) \geq \kappa \mid x \} \geq 1-\alpha \}$. Then $ \kappa_\alpha (x) = \sup S$. Recall that $\rho (\kappa \mid x) = P_{\Theta \mid X} \{ \pi_{\Theta \mid X}(\theta \mid x) > \kappa \mid x \}$. 
If $\kappa > \kappa_\alpha (x)$, then $\kappa \notin S$, so that $\rho (\kappa \mid x) \leq P_{\Theta \mid X} \{ \pi_{\Theta \mid X}(\theta \mid x) \geq \kappa \mid x \} < 1-\alpha$. By the right continuity of $\rho$, we have 
\begin{align}\label{eq:leq1-a}
    \rho \{ \kappa_\alpha (x) \mid x \} \leq 1-\alpha. 
\end{align}

Next, we show that the function $\kappa \mapsto \rho(\kappa- \mid x)$ is left continuous at $\kappa_\alpha (x)$ or equivalently, for any $\epsilon > 0$, there exists a $\delta > 0 $ such that for every $\kappa \in ( \kappa_\alpha (x)-\delta, \kappa_\alpha (x) )$, we have $\rho ( \kappa- \mid x ) - \rho \{ \kappa_\alpha (x)- \mid x \} < \epsilon$. By the definition of left limit, for any $\epsilon > 0 $, there exist a $\delta > 0$ such that $\kappa \in (\kappa_\alpha (x) - \delta, \kappa_\alpha (x))$ implies $\rho (\kappa \mid x) - \rho \{ \kappa_\alpha (x) - \mid x \} < \epsilon$. Now fix a $\kappa \in (\kappa_\alpha (x) - \delta, \kappa_\alpha (x))$ and let $\kappa^\prime \in (\kappa_\alpha (x) - \delta, \kappa )$. Then we have 
\begin{align*}
    \rho(\kappa^\prime \mid x) \geq \rho(\kappa- \mid x).
\end{align*}
Since $\kappa^\prime \in (\kappa_\alpha (x) - \delta, \kappa_\alpha (x) )$, we have $\rho(\kappa^\prime \mid x) - \rho \{ \kappa_\alpha (x)- \mid x \} < \epsilon
$ so that $\rho(\kappa - \mid x) - \rho \{ \kappa_\alpha (x)- \mid x \} < \epsilon$. Hence, $\rho(\kappa- \mid x)$ is left continuous at $\kappa_\alpha (x)$, and consequently, 
\begin{align}\label{eq:geq1-a}
    \rho \{ \kappa_\alpha (x)- \mid x \} \geq 1-\alpha.
\end{align}

Therefore, by \eqref{eq:leq1-a}, \eqref{eq:geq1-a}, and the continuity of $\rho$ at $\kappa_\alpha (x)$,  we obtain
\begin{align*}
    \rho \{ \kappa_\alpha (x) \mid x \} = \rho \{ \kappa_\alpha (x) - \mid x \} = 1-\alpha,
\end{align*}
which implies $P[ \{ \theta : \pi ( \theta \mid x) = \kappa_\alpha (x) \} \mid x ] = 0$.
Hence, we can take $\gamma(\theta \mid x) = 1$ in \eqref{eq:GHPD}, in which case
\begin{align*}
    E \{ \phi^* (\Theta \mid x) \mid x \} = P \{ \pi_{\Theta \mid X} ( \theta \mid x) \geq \kappa_\alpha (x) \mid x \}= \rho \{ \kappa_\alpha (x) - \mid x \} =1 - \alpha.
\end{align*}
Moreover,
\begin{align*}
    \{ \theta : \phi^* (\theta \mid x) =1 \} 
    = \{ \theta : \pi_{\Theta \mid X} ( \theta \mid x) \geq \kappa_\alpha (x)\}, 
\end{align*}
which is the highest posterior density credible set $C(\kappa_\alpha)$.    
\end{proof}

Figure~\ref{fig:normalnormalGHPD} shows that when $\theta$ is a continuous random variable and $\rho$ is continuous at $\kappa_\alpha (x)$, the $(1-\alpha)$-level generalized highest posterior density credible set $\phi^* (\theta \mid x )$ only takes values in $\{ 0, 1\}$, as claimed in Proposition \ref{prop:highest posterior density}. We can see that $\{\theta : \phi^* (\theta \mid x ) = 1 \}$ in Fig.~\ref{fig:normalGHPD} corresponds to the 95\% highest posterior density credible interval of $(-1.96, -1.96)$, where  $\Theta \mid X = x $ has $ \mathrm N(0, 1)$.
\begin{figure}
\centering
\subfloat[]{%
  \includegraphics[width=0.45\textwidth]{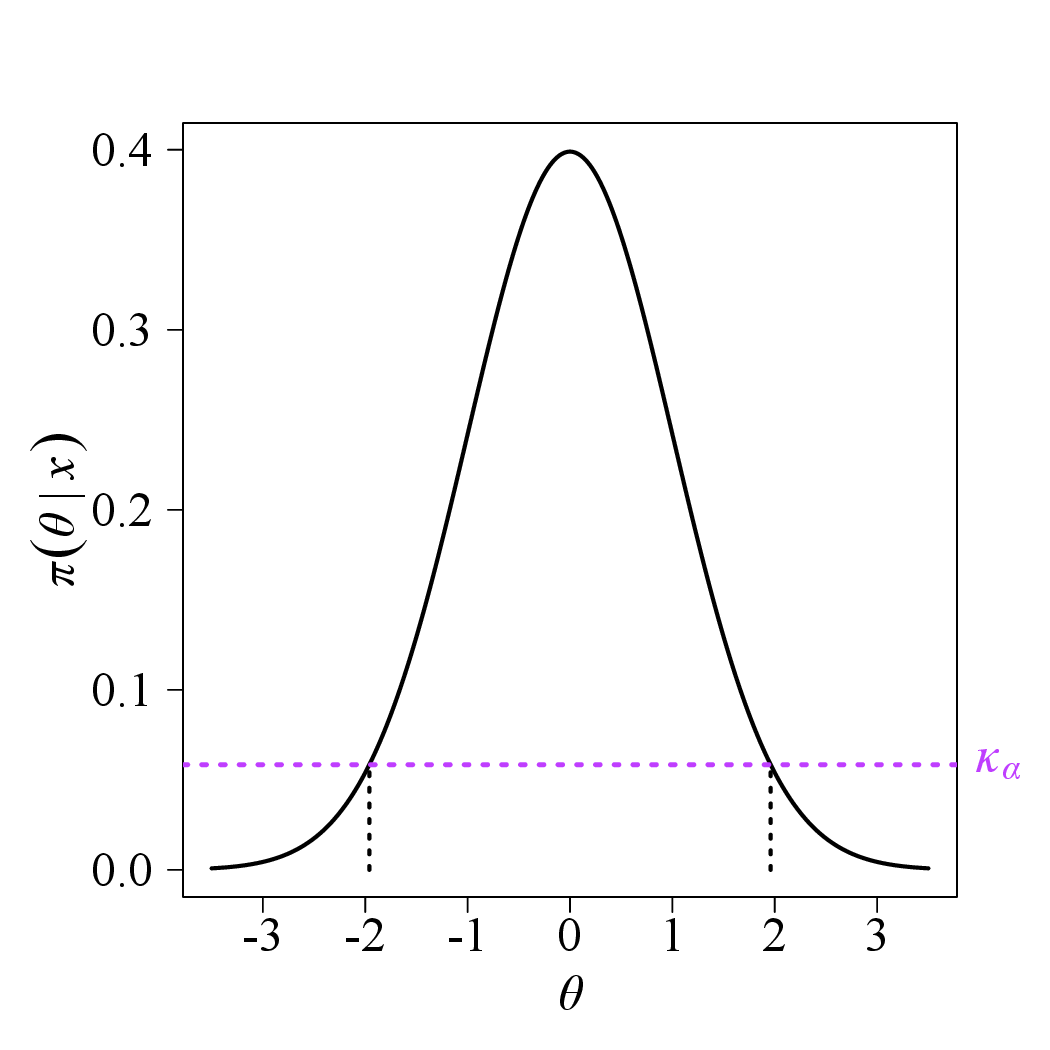}\label{fig:normal}%
}%
\subfloat[]{%
  \includegraphics[width=0.45\textwidth]{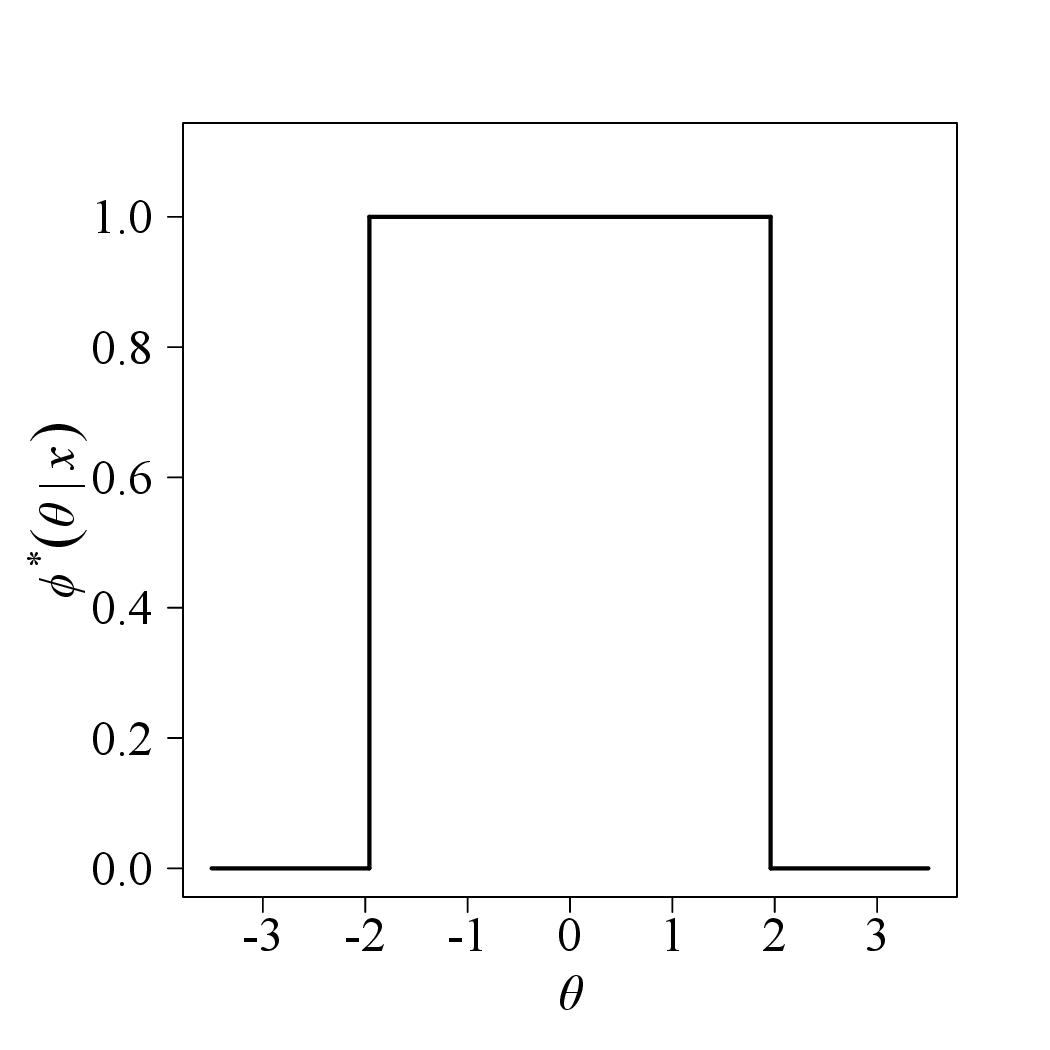}\label{fig:normalGHPD}%
}
\caption{(a) The 95\% highest posterior density credible set $(-1.96, 1.96)$ for the standard normal posterior density; (b) The 95\% generalized highest posterior density credible set for the standard normal posterior density.}
\label{fig:normalnormalGHPD}
\end{figure}

\section{Fair generalized highest posterior density credible set}

The generalized highest posterior density credible set $\phi^*$ in \eqref{eq:GHPD} is not unique, unless $P_{\Theta \mid X} \{ \pi_{\Theta \mid X}(\theta \mid x ) = \kappa_\alpha (x) \mid x \}=0$. When $P_{\Theta \mid X} \{ \pi_{\Theta \mid X}(\theta \mid x) = \kappa_\alpha (x) \mid x\} > 0$, any $\gamma (\theta \mid x)$ satisfying 
\begin{align}\label{eq:E}
    \int_{K(x)} \gamma (\theta \mid x) \pi_{\Theta \mid X} ( \theta \mid x) d \mu_\Theta (\theta) = 1 - \alpha - P_{\Theta \mid X} \{ \pi_{\Theta \mid X}(\theta \mid x) > \kappa_\alpha (x) \mid x \}
\end{align}
is optimal in the sense of Theorem \ref{thm:smallest} where $K(x) = \{ \theta: \pi_{\Theta \mid X} ( \theta \mid x) = \kappa_\alpha (x) \}$. In other words,  if we let 
\begin{align}\label{eq:GHPDcollection}
    S=\{ \phi^* : \text{ $\phi^*$ is of the form \eqref{eq:GHPD} satisfying \eqref{eq:E}} \},
\end{align}
then every member of $S$ is the generalized highest posterior density credible set. For definiteness, we introduce the fair generalized highest posterior density credible set as follows. 
\begin{definition}\label{def:fairGHPD}
Let $S$ be the collection of all generalized highest posterior density credible sets of level-$(1-\alpha)$ as defined in \eqref{eq:GHPDcollection}.
We call $\phi^*_f \in S$ the fair generalized highest posterior density credible set if 
\begin{align*}
    \var \{ \phi^*_f (\theta \mid x) \} \leq 
    \var \{ \phi^* (\theta \mid x) \}  \quad \text{for all $\phi^* \left(\theta \mid x\right) \in S$}. 
\end{align*}
\end{definition}

To understand this definition, it is helpful to interpret the value of $\phi^* (\theta \mid x)$ at $\theta$ as the probability of $\theta$ belonging to the credible set conditioning on $X = x$. That is, we toss a coin with head probability $\phi^* (\theta \mid x)$ to determine whether $\theta$ belongs to the credible set. This amounts to introducing a Bernoulli random variable $U$ at $\theta$ with a success probability of  $\phi^* (\theta \mid x)$. If $U=1$, then we decide $\theta$ belongs to the credible set, otherwise, it doesn't. So, when there is more than one generalized highest posterior density credible set, it is desirable to make $\phi^* (\theta \mid x)$ as similar as possible to give each $\theta$ a fair treatment. That is, we should minimize the variation of $\phi^* (\theta \mid x)$ with respect to $\pi_{\Theta \mid X} (\theta \mid x)$.

The above definition is needed only when $P_{\Theta \mid X} \{ \pi_{\Theta \mid X} ( \theta \mid x) = \kappa_\alpha (x) \} > 0$. When this probability is 0, it doesn't matter how we define $\gamma (\theta \mid x)$. For example, we can take $\gamma (\theta \mid x)=0$ on the set $\{ \theta: \pi_{\Theta \mid X} ( \theta \mid x) = \kappa_\alpha (x) \}$.
The following theorem characterizes the form of the fair generalized highest posterior density credible set.
\begin{figure}[]
\centering
\subfloat[]{%
  \includegraphics[width=0.45\textwidth]{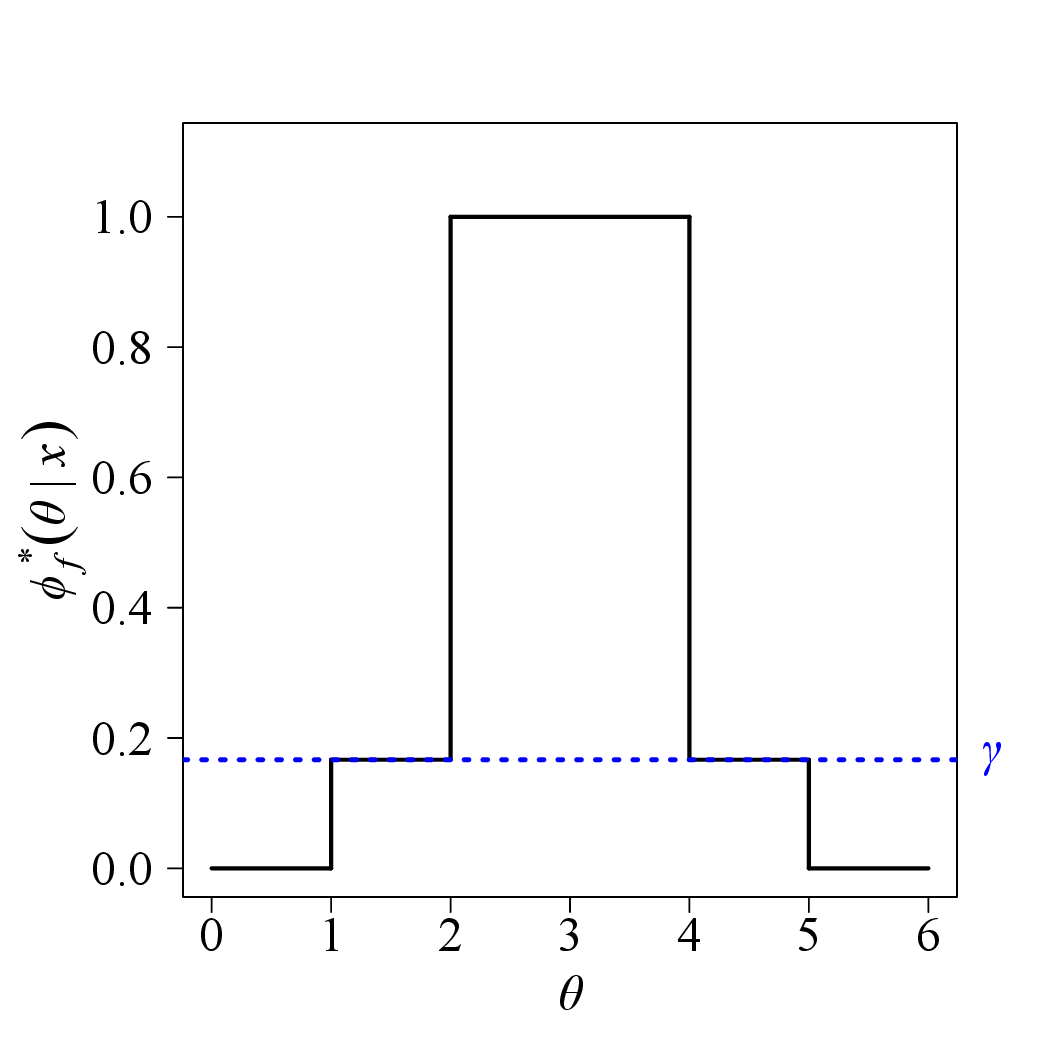}\label{fig:flatGHPD}%
}%
\subfloat[]{%
  \includegraphics[width=0.45\textwidth]{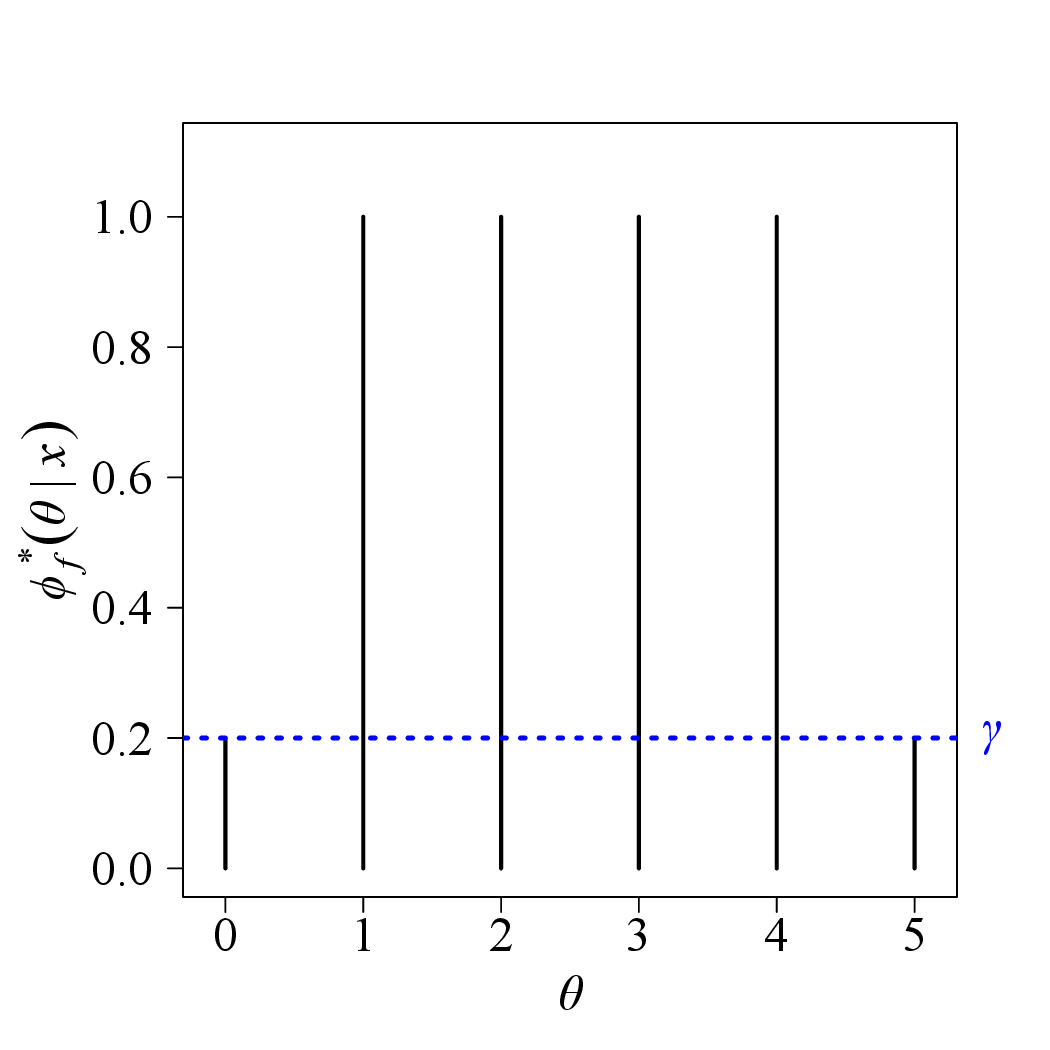}\label{fig:discGHPD}%
}
\caption{(a) The $(1-\alpha)$-level fair generalized highest posterior density credible set of Fig.~\ref{fig:flatdistn}; (b) The $(1-\alpha)$-level fair generalized highest posterior density credible set of Fig.~\ref{fig:discdistn}. 
}
\label{fig:flatdiscGHPD}
\end{figure}

\begin{theorem} \label{thm:gamma}
Suppose $P_{\Theta \mid X} \{ \pi_{\Theta \mid X} ( \theta \mid x) = \kappa_\alpha (x) \} > 0$. Then the fair generalized highest posterior density credible set $\phi^*_f$ is the member of $S$ with 
\begin{align*}
    \gamma (\theta \mid x)= 
    \cfrac{1 - \alpha - P_{\Theta \mid X} \{ \pi_{\Theta \mid X}(\theta \mid x) > \kappa_\alpha (x) \mid x \} }{P_{\Theta \mid X} \{ \pi_{\Theta \mid X}(\theta \mid x) = \kappa_\alpha (x) \mid x \} }.
\end{align*}
\end{theorem}

\begin{proof}
Let $K(x) = \{ \theta: \pi_{\Theta \mid X} ( \theta \mid x) = \kappa_\alpha (x) \}$.
Since $\phi^* \in S$ satisfies $E \{ \phi^*  (\Theta \mid X) \mid X \} = 1 - \alpha$, we have
\begin{align}\label{eq:constant}
    \int_{K(x)} \gamma (\theta \mid x) \pi_{\Theta \mid X} ( \theta \mid x) d \mu_\Theta (\theta) = 1 - \alpha - P_{\Theta \mid X} \{ \pi_{\Theta \mid X}(\theta \mid x) > \kappa_\alpha (x) \mid x \}.
\end{align}
By Definition \ref{def:fairGHPD}, the fair generalized highest posterior density credible set is 
\begin{align*}
    \phi^*_f = \argmin_{\phi^* \in S} [ \var \{ \phi^* (\Theta \mid X) \mid X \} ].
\end{align*}
Notice that
\begin{align*}
    \var \{ \phi^* (\Theta \mid X) \mid X \}
    & = E\{ \phi^* (\Theta \mid X)^2 \mid X \} - [ \{E \{ \phi^* (\Theta \mid X) \mid X\} ]^ 2\\
    & = P_{\Theta \mid X} \{ \pi_{\Theta \mid X}(\theta \mid x) > \kappa_\alpha (x) \mid x \} \\
    & \quad + \int_{K(x)} \gamma^2 (\theta \mid x) \pi_{\Theta \mid X} ( \theta \mid x) d \mu_\Theta (\theta) - (1-\alpha)^2 .
\end{align*}
Hence, it suffices to minimize
\begin{align*}
    \int_{K(x)} \gamma^2 (\theta \mid x) \pi_{\Theta \mid X} ( \theta \mid x) d \mu_\Theta (\theta) 
\end{align*}
subject to \eqref{eq:constant}. Since 
\begin{align*}
    \frac{\pi_{\Theta \mid X} ( \theta \mid x)}{P_{\Theta \mid X} \{ \pi_{\Theta \mid X} ( \theta \mid x) = \kappa_\alpha (x) \mid x \}}
\end{align*}
is a probability density on $K(x)$, we have, by Jensen's inequality, 
\begin{align*}
    & \int_{K(x)} \gamma^2 (\theta \mid x) \frac{\pi_{\Theta \mid X} ( \theta \mid x)}{P_{\Theta \mid X} \{ \pi_{\Theta \mid X} ( \theta \mid x) = \kappa_\alpha (x) \mid x \}} d \mu_\Theta (\theta) \\
    &\geq
    \left( \int_{K(x)} \gamma (\theta \mid x) \frac{\pi_{\Theta \mid X} ( \theta \mid x)}{P_{\Theta \mid X} \{ \pi_{\Theta \mid X} ( \theta \mid x) = \kappa_\alpha (x) \mid x \}} d \mu_\Theta (\theta) \right)^2.
\end{align*}
Canceling $P_{\Theta \mid X} \{ \pi_{\Theta \mid X} ( \theta \mid x) = \kappa_\alpha (x) \mid x \}$ on both sides above and evoking \eqref{eq:constant}, we have
\begin{align*}
    &\int_{K(x)} \gamma^2 (\theta \mid x) \pi_{\Theta \mid X} ( \theta \mid x) d \mu_\Theta (\theta)\\
    & \geq
    \frac{1}{P_{\Theta \mid X} \{ \pi_{\Theta \mid X} ( \theta \mid x) = \kappa_\alpha (x) \mid x \}}\left( \int_{K(x)} \gamma (\theta \mid x) \pi_{\Theta \mid X} ( \theta \mid x) d \mu_\Theta (\theta) \right)^2\\
    & = \frac{[1 - \alpha - P_{\Theta \mid X} \{ \pi_{\Theta \mid X} ( \theta \mid x) > \kappa_\alpha (x) \mid x \} ]^2}{P_{\Theta \mid X} \{ \pi_{\Theta \mid X} ( \theta \mid x) = \kappa_\alpha (x) \mid x \}}
\end{align*}
for any $\gamma (\theta \mid x)$. On the other hand, if we take 
\begin{align*}
    \gamma (\theta \mid x) = \frac{1 - \alpha - P_{\Theta \mid X} \{ \pi_{\Theta \mid X} ( \theta \mid x) > \kappa_\alpha (x) \mid x \} }{P_{\Theta \mid X} \{ \pi_{\Theta \mid X} ( \theta \mid x) = \kappa_\alpha (x) \mid x \}},
\end{align*}
then 
\begin{align*}
    &\int_{K(x)} \gamma^2 (\theta \mid x) \pi_{\Theta \mid X} ( \theta \mid x) d \mu_\Theta (\theta)\\
    & =
    \left[\frac{1 - \alpha - P_{\Theta \mid X} \{ \pi_{\Theta \mid X} ( \theta \mid x) > \kappa_\alpha (x) \mid x \}}{P_{\Theta \mid X} \{ \pi_{\Theta \mid X} ( \theta \mid x) = \kappa_\alpha (x) \mid x \}} \right]^2 \int_{K(x)} \pi_{\Theta \mid X} ( \theta \mid x) d \mu_\Theta (\theta)\\
    & =
    \left[\frac{1 - \alpha - P_{\Theta \mid X} \{ \pi_{\Theta \mid X} ( \theta \mid x) > \kappa_\alpha (x) \mid x \}}{P_{\Theta \mid X} \{ \pi_{\Theta \mid X} ( \theta \mid x) = \kappa_\alpha (x) \mid x \}} \right]^2 P_{\Theta \mid X} \{ \pi_{\Theta \mid X} ( \theta \mid x) = \kappa_\alpha (x) \mid x \}\\
    & = \frac{[1 - \alpha - P_{\Theta \mid X} \{ \pi_{\Theta \mid X} ( \theta \mid x) > \kappa_\alpha (x) \mid x \} ]^2}{P_{\Theta \mid X} \{ \pi_{\Theta \mid X} ( \theta \mid x) = \kappa_\alpha (x) \mid x \}}, 
\end{align*}
which achieves the lower bound.
\end{proof}

To provide more intuition about our construction, let us revisit the scenario in Fig.~\ref{fig:highest posterior densityfail} where $(1-\alpha)$-level highest posterior density credible set fails to exist and see how the fair generalized highest posterior density credible set fixes the problem. Figure \ref{fig:flatdiscGHPD} illustrates that the $(1-\alpha)$-level fair generalized highest posterior density credible set $\phi^*_f$ achieves the exact $(1-\alpha)$-level credibility by having the compensating constant $\gamma (\theta \mid x)$. The value of $\gamma (\theta \mid x)$ is 0.167 for Fig.~\ref{fig:flatGHPD} and 0.2 for Fig.~\ref{fig:discGHPD}. Keep in mind that $\phi^*_f$ is no longer an indicator function of a set, and when $\phi^*_f$ takes value in $(0,1)$, it represents the probability that $\theta$ belongs to a particular class given the data $X$.

\section{The Steering Wheel Plot}

We now develop a graphical representation of the fair generalized highest posterior density credible set in the classification setting, which is called the steering wheel plot, which helps to visualize the uncertainty in classification.

The steering wheel plot consists of an inner hub, spokes, and an outer rim as shown in Fig.~\ref{fig:swp}. The inner hub represents the estimated class in color. The spokes represent the $\phi_f^* (\theta \mid x )$ values. The color of a spoke symbolizes the class it represents. Corresponding to the value of $\phi_f^* (\theta \mid x )$, the length of a spoke can be $0$, $\gamma(\theta \mid x)$, or $1$. In this figure,  $\gamma(\theta \mid x)$ is abbreviated as $\gamma$. The set $\{\theta: \phi_f^*  (\theta \mid x ) = 1 \}$ is called the interior of the fair generalized highest posterior density credible set, the set $\{\theta: \phi_f^*  (\theta \mid x ) = 0\}$ the exterior, and the set $\{\theta: \phi_f^*  (\theta \mid x ) = \gamma \}$ the boundary.

\begin{figure}
\centering
    \centering
    \includegraphics[width=0.45\textwidth]{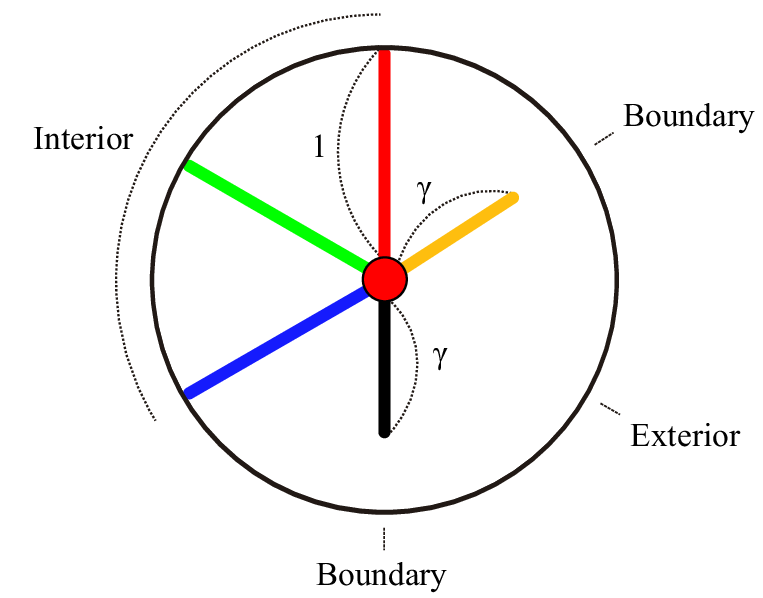}
    \caption{Steering wheel plot structure}
    \label{fig:swp}
\end{figure}

\section{Simulation and data application}
\subsection{Simulation}
In this section, we demonstrate how the steering wheel plot can assist in the visualization of uncertainty in classification problems. Suppose $X$ is a 2-dimensional random vector having a bivariate normal distribution, and for each class $\theta \in \{ \text{red, green, blue}\}$, $X$ has a mean dependent on $\theta$, but a variance matrix independent of $\theta$. We generate $X$ from
\begin{align*}
    X_{\theta i} \sim N _2( \mu_\theta,  \Sigma), 
    \quad ( \theta=\text{red, green, blue};\ i = 1, \ldots, 10),
\end{align*}
where $ \mu_\text{red} = (5, 6)^\T$, $ \mu_\text{green} = (4, 5)^\T$, $ \mu_\text{blue} = (6, 4)^\T$ and $ \Sigma =  I_2$. We apply quadratic discriminant analysis to estimate posterior density $\pi_{\Theta \mid X} (\theta \mid x)$ using the generated data and then classify the data points into the class which has the largest posterior probability. We then obtain a 95\%-level fair generalized highest posterior density credible set $\phi^*_f$ and present the steering wheel plot for each classified object. 

In Fig.~\ref{fig:sim1}, we see that a steering wheel plot has a single spoke when the classification is certain but more spokes when the classification becomes less certain. For instance, let us compare the red data points A and B in Fig.~\ref{fig:sim1original}. Point A appears to belong to the red class with certainty, and thus, it has only one red spoke with a length of 0.967 in Fig.~\ref{fig:sim1steer}. This suggests that if the red class is included in the fair generalized highest posterior credible set with a probability of 0.967, then the credible set attains the 95\% credible level. On the other hand, point B seems to be more difficult to classify, and it is even classified as blue with three spokes in Figure~\ref{fig:sim1steer}. The spoke lengths are 1 for red and blue, and 0.49 for green. The 95\% fair generalized credible set for that point contains red and blue classes for sure but contains the green class with a probability of 0.49.

\begin{figure}[]
\centering
\subfloat[]{%
  \includegraphics[width=0.45\textwidth]{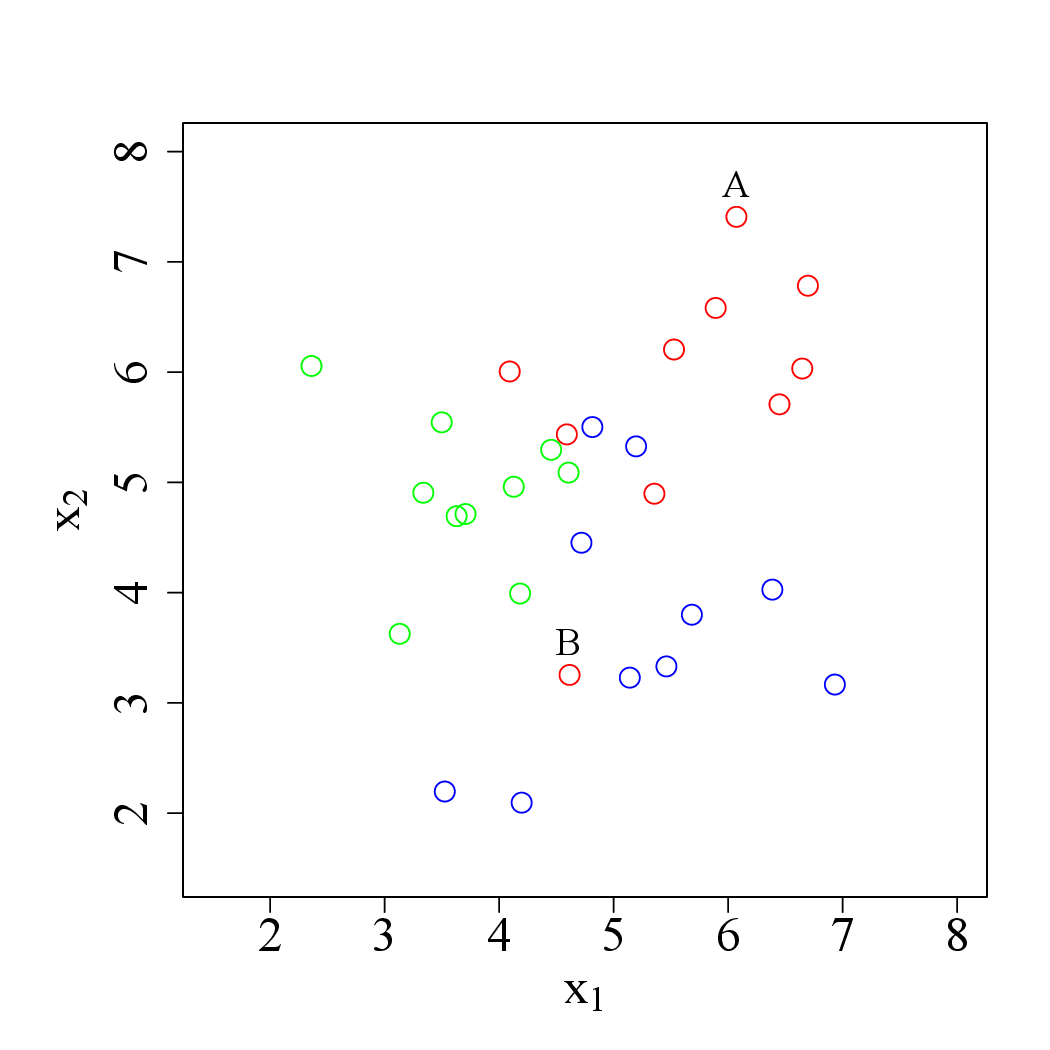}\label{fig:sim1original}%
}%
\subfloat[]{%
  \includegraphics[width=0.45\textwidth]{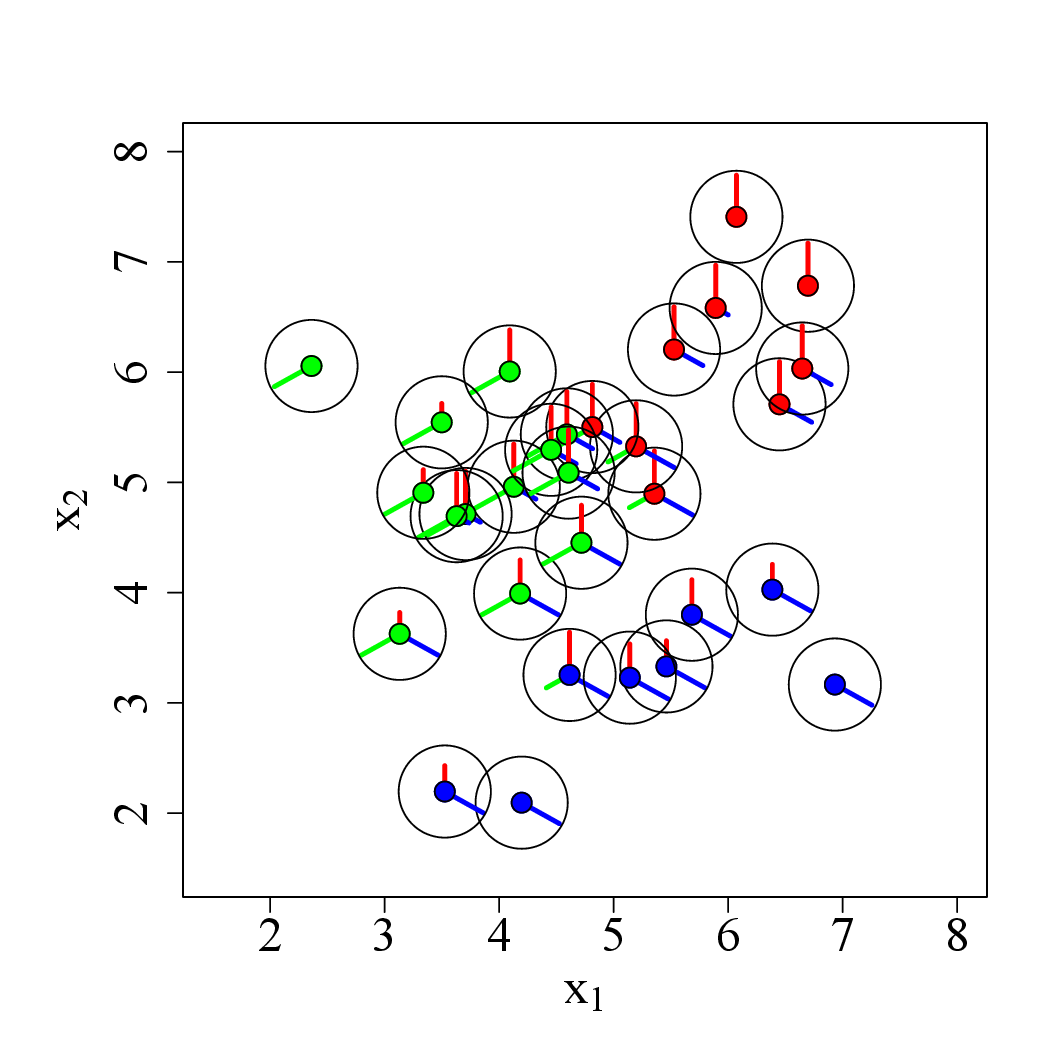}\label{fig:sim1steer}%
}
\caption{Three classes are represented by red, green, or blue. (a) Generated data with true labels; (b) Classification results with steering wheel plots.}
\label{fig:sim1}
\end{figure}

\subsection{Data application}
We now apply the fair generalized highest posterior density credible set and the steering wheel plot to speaker accent recognition data. It is available at the UCI Machine Learning Repository at \url{https://archive.ics.uci.edu/ml/datasets/Speaker+Accent+Recognition#}; see also \citet{ma2015comparison}. There are 12 integer features and six different accents of English: Spanish, French, German, Italian, British, and American. The data has a sample size of 329. Speakers from six different countries read single English words. Twelve features for each subject are extracted from the soundtrack of no more than one second of reading a word using Mel-Frequency Cepstral Coefficients. 

For better visual effect, we first perform sufficient dimension reduction to reduce the dimension of the predictors \citep{li2018sufficient}. 
We apply sliced inverse regression \citep{li1991sliced} and directional regression \citep{li2007directional} to find low-dimensional predictors with four dimensions. We then conduct quadratic discriminant analysis to classify the accents in the reduced space. The 95\%-level fair generalized highest posterior density credible set and steering wheel plot enable us to infer the classification uncertainty.

Figure \ref{fig:accent} shows the distinction between the classification uncertainties of two different quadratic discriminant analysis classifiers. For example, considering the French accent, the spokes of the classifier with sliced inverse regression are French, American, Spanish, or British, whereas the spokes of the classifier with directional regression are French or American. We can infer that directional regression helps to reduce classification uncertainty when classifying French accents.
\begin{figure}[]
\centering
\subfloat[]{%
  \includegraphics[width=0.45\textwidth]{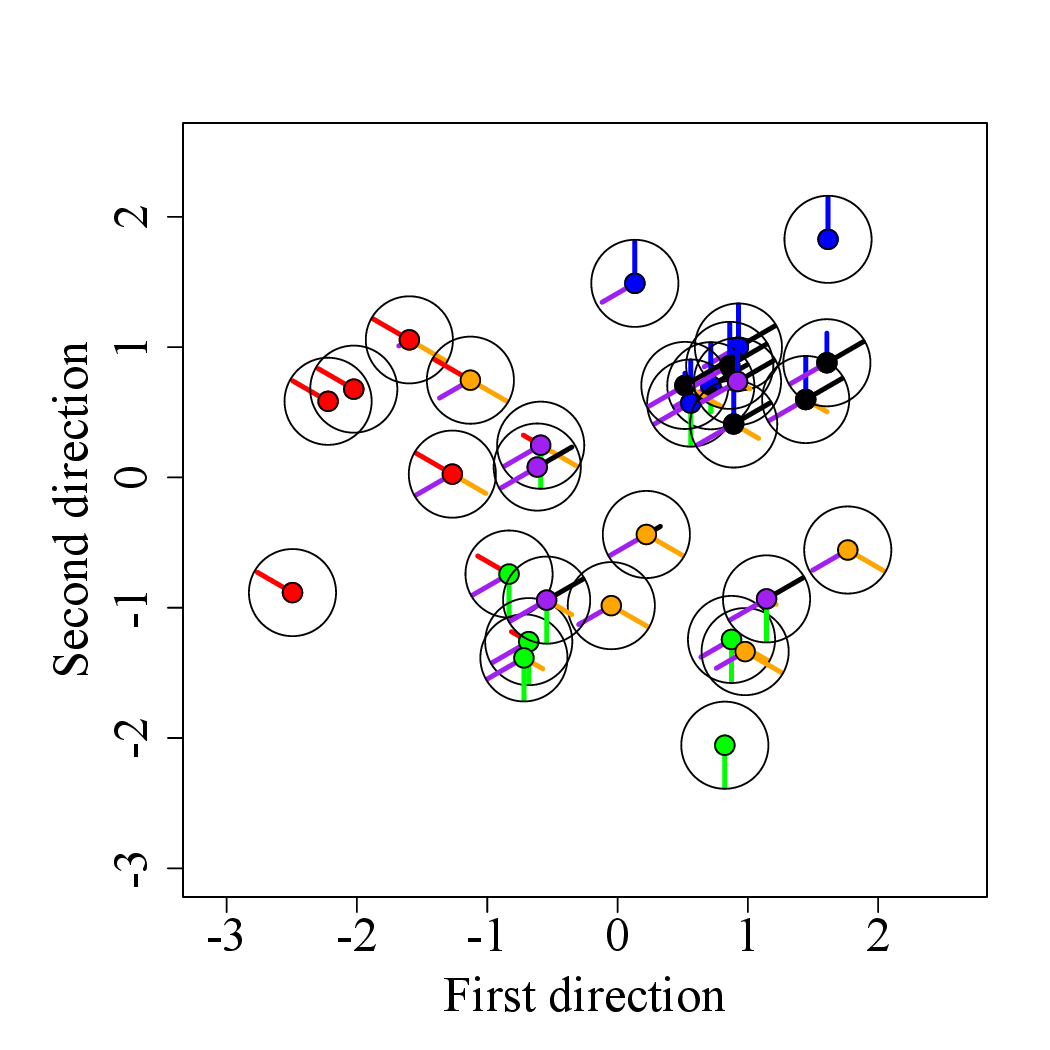}\label{fig:accent_SIR}%
}%
\subfloat[]{%
  \includegraphics[width=0.45\textwidth]{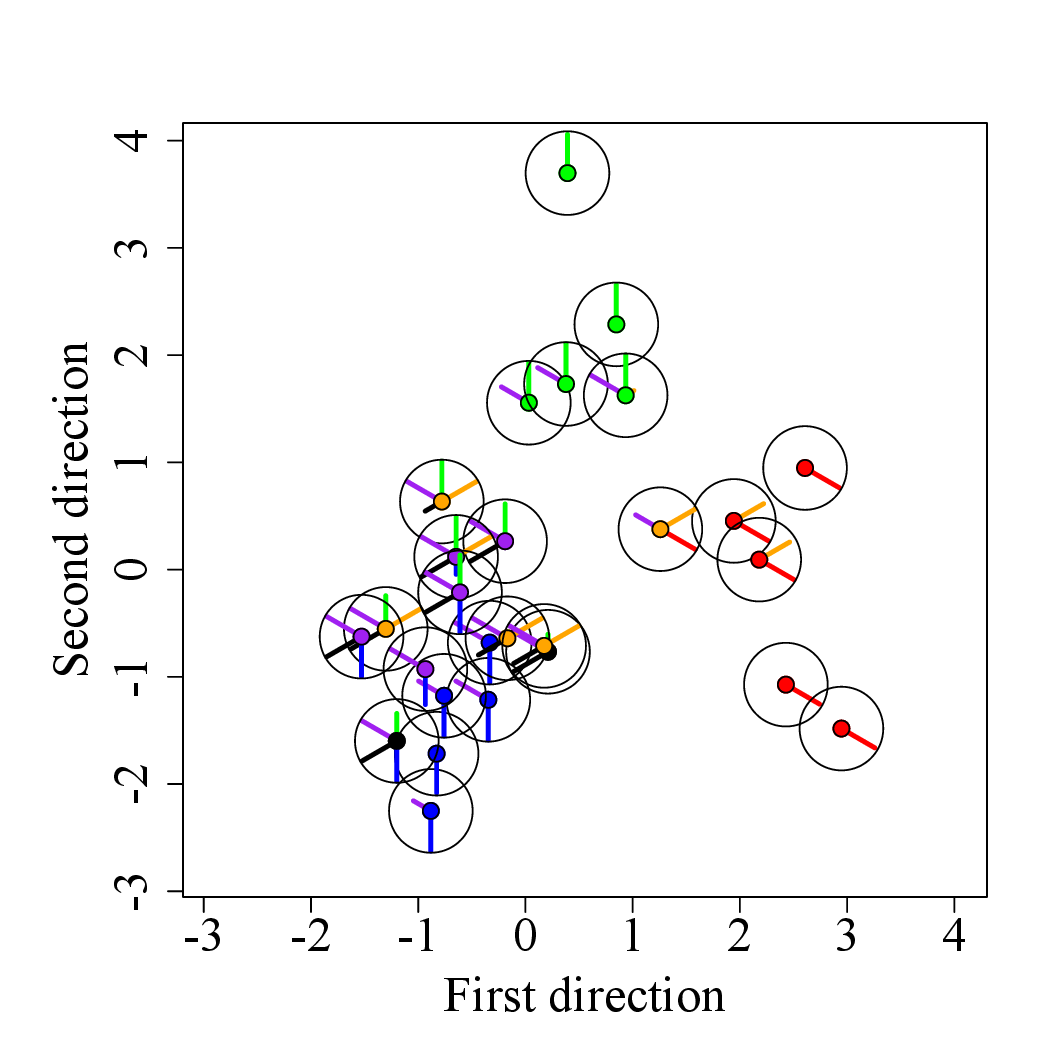}\label{fig:accent_DR}%
}
\caption{Speaker accent data: the steering wheel plots of the first two sufficient predictors. Spanish is red, French is green, German is blue, Italian is black, British is orange, and American is purple. To enhance clarity, five samples are randomly selected for each classified accent. (a) sliced inverse regression predictors; (b) directional regression predictors.}
\label{fig:accent}
\end{figure}

\section*{Supplementary material}
\label{SM}
Supplementary material includes the code for reproducing results and plots represented in this paper. The code is available at \url{https://github.com/Chaegeun/ExactBayesianCredibleSet}.

\bibliographystyle{asa}
\bibliography{paper-ref}

\end{document}